\documentclass{article}%
\usepackage{amssymb}
\usepackage{amsmath}
\usepackage{amsfonts}
\usepackage{geometry}%
\usepackage{graphicx}
\providecommand{\U}[1]{\protect\rule{.1in}{.1in}}
\newtheorem{theorem}{Theorem}[section]

\newtheorem{lemma}[theorem]{Lemma}

\newtheorem{proposition}[theorem]{Proposition}

\newenvironment{proof}[1][Proof]{\noindent\textbf{#1.} }{\ \rule{0.5em}{0.5em}}
\begin{document}

\title{On Tropical Linear and Integer Programs}
\author{Peter Butkovi\v{c}\thanks{E-mail: p.butkovic@bham.ac.uk}\\School of Mathematics, University of Birmingham\\Birmingham B15 2TT, United Kingdom}
\maketitle

\begin{abstract}
We present simple compact proofs of the strong and weak duality theorems of
tropical linear programming. It follows that there is no duality gap for a
pair of tropical primal-dual problems. This result together with known
properties of subeigenvectors enables us to directly solve a special tropical
linear program with two-sided constraints.

We also study the duality gap in tropical integer linear programming. A direct
solution is available for the primal problem. An algorithm of quadratic
complexity is presented for the dual problem. A direct solution is available
provided that all coefficients of the objective function are integer. This
solution provides a good estimate of the optimal objective function value in
the general case.

AMS classification: 15A18, 15A80

Keywords: tropical linear programming; tropical integer programming; duality; residuation.

\end{abstract}

\section{\bigskip Introduction}

Tropical linear algebra (also called max-algebra or path algebra) is an
analogue of linear algebra developed for the pair of operations $\left(
\oplus,\otimes\right)  $ where%
\[
a\oplus b=\max(a,b)
\]
and
\[
a\otimes b=a+b
\]
for $a,b\in\overline{\mathbb{R}}\overset{def}{=}\mathbb{R}\cup\{-\infty\}. $
This pair is extended to matrices and vectors as in conventional linear
algebra. That is if $A=(a_{ij}),~B=(b_{ij})$ and $C=(c_{ij})$ are matrices of
compatible sizes with entries from $\overline{\mathbb{R}}$, we write
$C=A\oplus B$ if $c_{ij}=a_{ij}\oplus b_{ij}$ for all $i,j$ and $C=A\otimes B$
if
\[
c_{ij}=\bigoplus\limits_{k}a_{ik}\otimes b_{kj}=\max_{k}(a_{ik}+b_{kj})
\]
for all $i,j$. If $\alpha\in\overline{\mathbb{R}}$ then $\alpha\otimes
A=\left(  \alpha\otimes a_{ij}\right)  $. For simplicity we will use the
convention of not writing the symbol $\otimes.$ Thus in what follows the
symbol $\otimes$ will not be used (except when necessary for clarity), and
unless explicitly stated otherwise, all multiplications indicated are in max-algebra.

The interest in tropical linear algebra \bigskip was originally motivated by
the possibility of dealing with a class of non-linear problems in pure and
applied mathematics, operational research, science and engineering as if they
were linear due to the fact that $\left(  \overline{\mathbb{R}},\oplus
,\otimes\right)  $ is a commutative and idempotent semifield. Besides the main
advantage of using linear rather than non-linear techniques, tropical linear
algebra enables us to efficiently describe and deal with complex sets \cite{PB
BMIS}, reveal combinatorial aspects of problems \cite{Linalgcomb} and view a
class of problems in a new, unconventional way. The first pioneering papers
appeared in the 1960s \cite{CG first}, \cite{CG0} and \cite{Vorobjov},
followed by substantial contributions in the 1970s and 1980s such as
\cite{CG1}, \cite{GM dioids}, \cite{KZimm} and \cite{Cohen et al}. Since 1995
we have seen a remarkable expansion of this research field following a number
of findings and applications in areas as diverse as algebraic geometry
\cite{mikhalkin} and \cite{sturmfels}, geometry \cite{joswig}, control theory
and optimization \cite{baccelli}, phylogenetic \cite{speyer sturmfels},
modelling of the cellular protein production \cite{Brackley} and railway
scheduling \cite{heidergott et al}. A number of research monographs have been
published \cite{baccelli}, \cite{PB book}, \cite{heidergott et al} and
\cite{McEneaney}. A chapter on max-algebra appears in a handbook of linear
algebra \cite{hogben} and a chapter on idempotent semirings is in a monograph
on semirings \cite{golan}.

Tropical linear algebra covers a range of linear-algebraic problems in the
max-linear setting, such as systems of linear equations and inequalities,
linear independence and rank, bases and dimension, polynomials, characteristic
polynomials, matrix equations, matrix orbits and periodicity of matrix powers
\cite{baccelli}, \cite{PB book}, \cite{CG1} \cite{heidergott et al}. Among the
most intensively studied questions was the \textit{eigenproblem}, that is the
question, for a given square matrix $A$ to find all values of $\lambda$ and
non-trivial vectors $x$ such that $Ax=\lambda x.$ This and related questions
such as $z$-matrix equations $Ax\oplus b=\lambda x$ \cite{BSS} have been
answered \cite{CG1}, \cite{GM dioids}, \cite{gaubert thesis}, \cite{bapat
first}, \cite{PB book} with numerically stable low-order polynomial
algorithms. The same is true about the \textit{subeigenproblem} that is
solution to $Ax\leq\lambda x,$ which appears to be strongly linked to the
eigenproblem. In contrast, attention has only recently been paid to the
\textit{supereigenproblem} that is solution to $Ax\geq\lambda x,$ which is
trivial for small values of $\lambda$ but in general the description of the
whole solution set seems to be much more difficult than for the eigenproblem
\cite{super ev}, \cite{ss supereig}. At the same time tropical linear and
integer linear programs have also been studied \cite{KZimm}, \cite{PB book},
\cite{MLP}, \cite{gks}, \cite{PB+Marie}. While one-sided tropical linear
systems of equations and inequalities are solvable in low-order polynomial
time, no polynomial method seems to exist for solving their two-sided
counterparts in general. Consequently, linear programs with one-sided
constraints are easily solvable while polynomial solution method for linear or
integer linear programs with two-sided constraints remains an open question.

The aim of this paper is to give simple compact proofs of duality theorems for
tropical linear programs with one-sided constraints and then use this result
to present efficient solution methods for solving

(a) a special type of tropical linear programs with two-sided constraints, and

\bigskip(b) tropical dual integer programs.

Note that the weak and strong duality in tropical linear programming has been
investigated in the past, in various settings \cite{hoffman}, \cite{UZimm}.
Duality theorems in the present paper appeared in \cite{superville} with
rather complicated proofs. The present paper will use new methodology, not
available in the 1970s, to provide simple and compact proofs.

Consider the following motivational example \cite{CG1}. Products
$P_{1},...,P_{m}$ are prepared using $n$ machines (or processors), every
machine contributing to the completion of each product by producing a partial
product. It is assumed that each machine can work for all products
simultaneously and that all these actions on a machine start as soon as the
machine starts to work. Let $a_{ij}$ be the duration of the work of the
$j^{th}$ machine needed to complete the partial product for $P_{i}$
$(i=1,...,m;j=1,...,n).$ If this interaction is not required for some $i$ and
$j$ then $a_{ij}$ is set to $-\infty.$ The matrix $A=\left(  a_{ij}\right)  $
is called the \textit{production matrix}. Let us denote by $x_{j}$ the
starting time of the $j^{th}$ machine $(j=1,...,n)$. Then all partial products
for $P_{i}$ $(i=1,...,m)$ will be ready at time
\[
\max(x_{1}+a_{i1},...,x_{n}+a_{in}).
\]
Hence if $b_{1},...,b_{m}$ are given completion times then the starting times
have to satisfy the system of equations:%

\[
\max(x_{1}+a_{i1},...,x_{n}+a_{in})=b_{i}\text{ \ for all }i=1,...,m.
\]
Using max-algebra this system can be written in a compact form as a system of
linear equations:%
\begin{equation}
Ax=b.\label{one-sided}%
\end{equation}
A system of the form (\ref{one-sided}) is called a \textit{one-sided system of
max-linear equations} (or briefly a \textit{one-sided max-linear system }or
just a\textit{\ max-linear system}). Such systems are easily solvable \cite{CG
first}, \cite{PB book}, \cite{KZimm}, see also Section \ref{Sec Prelim}.

In applications it may be required that the starting times are as small as
possible (to reflect the fact that the producer has free capacity and wishes
to start production as soon as possible). This means to minimise the value of%
\begin{equation}
f(x)=\max(x_{1},...,x_{n})\label{eq    f(x) is otx}%
\end{equation}
with respect to (\ref{one-sided}). In general we may need to minimize (or
possibly to maximize) a \textit{max-linear }function, that is,
\begin{equation}
f(x)=f^{T}x=\max(f_{1}+x_{1},...,f_{n}+x_{n}),\label{eq   mlr function}%
\end{equation}
where $f=\left(  f_{1},...,f_{n}\right)  ^{T}.$ So in (\ref{eq f(x) is otx})
we have $f=\left(  0,...,0\right)  ^{T}.$ Thus the max-linear programs are of
the form%
\begin{equation}
f^{T}x\longrightarrow\min\text{ or }\max\nonumber
\end{equation}
subject to%
\[
Ax=b.
\]

Sometimes the vector $b$ is given to require the earliest or latest rather
than exact completion times. In such cases the constraints of the optimization
problem are%
\[
Ax\geq b
\]
or,%
\[
Ax\leq b.
\]
In some cases two processes with the same starting times and production
matrices, say, $A$ and $B,$ have to be co-ordinated so that the products of
the second are completed not before those completed by the first and possibly
also not before given time restrictions given by a vector $d=\left(
d_{1},...,d_{m}\right)  ^{T}.$ Then the constraints have the form%
\begin{equation}
Ax\oplus d\leq Bx.\label{eq   tsls special}%
\end{equation}
In full generality, max-linear programs are of the form%
\[
f^{T}x\longrightarrow\min\text{ or }\max
\]
s.t.%
\[
Ax\oplus c=Bx\oplus d.
\]

\bigskip

These have been first studied in \cite{MLP}, with pseudopolynomial methods
presented in \cite{akian equivalence to mean payoff}, \cite{allamigeon et al}
and \cite{gks}. No polynomial solution method seems to exist in general, not
even for feasibility checking, although it is known that this problem in
$NP\cap co-NP$ \cite{bezem}.

\section{ Definitions, notation and preliminary results \label{Sec Prelim}}

Throughout the paper we denote $-\infty$ by $\varepsilon$ (the neutral element
with respect to $\oplus$) and for convenience we also denote by the same
symbol any vector, whose all components are $-\infty,$ or a matrix whose all
entries are $-\infty.$ A matrix or vector with all entries equal to $0$ will
also be denoted by $0.$ If $a\in\mathbb{R}$ then the symbol $a^{-1}$ stands
for $-a.$ Matrices and vectors whose all entries are real numbers are called
\textit{finite}. We assume everywhere that $n,m\geq1$ are integers and denote
$N=\left\{  1,...,n\right\}  ,M=\left\{  1,...,m\right\}  .$

It is easily proved that if $A,B,C$ and $D$ are matrices of compatible sizes
(including vectors considered as $m\times1$ matrices) then the usual laws of
associativity and distributivity hold and also isotonicity is satisfied:%
\begin{equation}
A\geq B\Longrightarrow AC\geq BC\text{ \ and }DA\geq DB.\label{Isotonicity}%
\end{equation}

\bigskip

A square matrix is called \textit{diagonal} if all its diagonal entries are
real numbers and off-diagonal entries are $\varepsilon.$ More precisely, if
$x=\left(  x_{1},...,x_{n}\right)  ^{T}\in\mathbb{R}^{n}$ then $diag\left(
x_{1},...,x_{n}\right)  $ or just $diag\left(  x\right)  $ is the $n\times n$
diagonal matrix%
\[
\left(
\begin{array}
[c]{cccc}%
x_{1} & \varepsilon & ... & \varepsilon\\
\varepsilon & x_{2} & ... & \varepsilon\\
\vdots & \vdots & \ddots & \vdots\\
\varepsilon & \varepsilon & ... & x_{n}%
\end{array}
\right)  .
\]
The matrix $diag\left(  0\right)  $ is called the \textit{unit matrix} and
denoted $I.$ Obviously, $AI=IA=A$ whenever $A$ and $I$ are of compatible
sizes. A matrix obtained from a diagonal matrix [unit matrix] by permuting the
rows and/or columns is called a \textit{generalized permutation matrix}
[\textit{permutation matrix}]. It is known that in tropical linear algebra
generalized permutation matrices are the only type of invertible matrices
\cite{CG1}, \cite{PB book}. Clearly,
\[
\left(  diag\left(  x_{1},...,x_{n}\right)  \right)  ^{-1}=diag\left(
x_{1}^{-1},...,x_{n}^{-1}\right)  .
\]
If $A$ is a square matrix then the iterated product $AA...A$ in which the
symbol $A$ appears $k$-times will be denoted by $A^{k}$. By definition
$A^{0}=I$.

\bigskip Given $A\in\overline{\mathbb{R}}^{n\times n}$ it is usual \cite{CG1},
\cite{baccelli}, \cite{heidergott et al}, \cite{PB book} in max-algebra to
define the infinite series
\begin{equation}
A^{\ast}=I\oplus A^{+}=I\oplus A\oplus A^{2}\oplus A^{3}\oplus...\text{
.}\label{Def Delta}%
\end{equation}
The matrix $A^{\ast}$ is called the \textit{strong transitive closure} of $A,$
or the \textit{Kleene Star}.

It follows from the definitions that every entry of the matrix sequence
\[
\left\{  A\oplus A^{2}\oplus...\oplus A^{k}\right\}  _{k=0}^{\infty}%
\]
is a nondecreasing sequence in $\overline{\mathbb{R}}$ and therefore either it
is convergent to a real number (when bounded) or its limit is $+\infty$ (when
unbounded). If $\lambda(A)\leq0$ then%
\[
A^{\ast}=I\oplus A\oplus A^{2}\oplus...\oplus A^{k-1}%
\]
for every $k\geq n$ and can be found using the Floyd-Warshall algorithm in
$O\left(  n^{3}\right)  $ time \cite{PB book}.

\bigskip The matrix\ $\lambda^{-1}A$ for $\lambda\in\mathbb{R}$ will be
denoted by $A_{\lambda}$ and $(A_{\lambda})^{\ast}$ will be shortly written as
$A_{\lambda}^{\ast}.$

Given $A=(a_{ij})\in\overline{\mathbb{R}}^{n\times n}$ the symbol $D_{A}$ will
denote the weighted digraph $\left(  N,E,w\right)  $ (called
\textit{associated with} $A$) where $E=\left\{  \left(  i,j\right)
;a_{ij}>\varepsilon\right\}  $ and $w\left(  i,j\right)  =a_{ij}$ for all
$(i,j)\in E.$ The symbol $\lambda(A)$ will stand for the \textit{maximum cycle
mean} of $A$, that is:
\begin{equation}
\lambda(A)=\max_{\sigma}\mu(\sigma,A),\label{mcm}%
\end{equation}
where the maximization is taken over all elementary cycles in $D_{A},$ and
\begin{equation}
\mu(\sigma,A)=\frac{w(\sigma,A)}{l\left(  \sigma\right)  }\label{mean}%
\end{equation}
denotes the \textit{mean} of a cycle $\sigma$. With the convention
$\max\emptyset=\varepsilon$ the value $\lambda\left(  A\right)  $ always
exists since the number of elementary cycles is finite. It can be computed in
$O\left(  n^{3}\right)  $ time \cite{karp}, see also \cite{PB book}. Observe
that $\lambda\left(  A\right)  =\varepsilon$ if and only if $D_{A}$ is acyclic.

\bigskip The tropical \textit{eigenvalue-eigenvector problem }(briefly
\textit{eigenproblem}) is the following:

\textit{Given }$A\in\overline{\mathbb{R}}^{n\times n}$\textit{, find all
}$\lambda\in\overline{\mathbb{R}}$ \textit{(eigenvalues) and\ }$x\in
$\textit{\ }$\overline{\mathbb{R}}^{n},x\neq\varepsilon$%
\textit{\ (eigenvectors) such that }%
\[
Ax=\lambda x.
\]

This problem has been studied since the work of R.A.Cuninghame-Green
\cite{CG0}. An $n\times n$ matrix has up to $n$ eigenvalues with
$\lambda\left(  A\right)  $ always being the largest eigenvalue (called
\textit{principal}). This finding was first presented by R.A.Cuninghame-Green
\cite{CG1} and M.Gondran and M.Minoux \cite{gondran minoux}, see also
N.N.Vorobyov \cite{Vorobjov}. The full spectrum was first described by
S.Gaubert \cite{gaubert thesis} and R.B.Bapat, D.Stanford and P. van den
Driessche \cite{bapat first}. The spectrum and bases of all eigenspaces can be
found in $O(n^{3})$ time \cite{Robustness paper} and \cite{PB book}.

The aim of this paper is to study integer solutions to tropical linear
programs and therefore we summarize here only the results on finite solutions
and for finite $A$ and $b$.

\begin{theorem}
\label{Th Ray finiteness crit}\cite{CG0}, \cite{CG1}, \cite{gondran minoux}
\bigskip If $A\in\mathbb{R}^{n\times n}$ then $\lambda\left(  A\right)  $ is
the unique eigenvalue of $A$ and all eigenvectors of $A$ are finite.
\end{theorem}

If $A\in\mathbb{R}^{n\times n}$ and $\lambda\in\mathbb{R}$ then a vector
$x\in\mathbb{R}^{n},x\neq\varepsilon$ satisfying
\begin{equation}
Ax\leq\lambda x\label{Spectral ineq}%
\end{equation}
is called a \textit{subeigenvector of }$A$\textit{\ with associated
subeigenvalue }$\lambda$\textit{. W}e denote $V_{\ast}(A,\lambda)=\left\{
x\in\mathbb{R}^{n}:Ax\leq\lambda x\right\}  .$

\begin{theorem}
\label{Th Set of subeigenvectors}\cite{BS}, \cite{PB book} Let $A\in
\mathbb{R}^{n\times n}$. Then $V_{\ast}\left(  A,\lambda\right)  \neq
\emptyset$ if and only if $\lambda\geq\lambda\left(  A\right)  $ and $V_{\ast
}\left(  A,\lambda\right)  =\left\{  A_{\lambda}^{\ast}u:u\in\mathbb{R}%
^{n}\right\}  .$
\end{theorem}

Let us define \cite{CG1}, \cite{PB book} min-algebra over $\mathbb{R}$ by%
\[
a\oplus^{\prime}b=\min(a,b)
\]
and
\[
a\otimes^{\prime}b=a\otimes b
\]
for all $a$ and $b.$ We extend the pair of operations $\left(  \oplus^{\prime
},\otimes^{\prime}\right)  $ to matrices and vectors in the same way as in max-algebra.

\bigskip We also define the \textit{conjugate} $A^{\#}=-A^{T}.$ It is easily
seen that
\begin{equation}
\left(  A^{\#}\right)  ^{\#}=A,\label{A hash hash}%
\end{equation}
\begin{equation}
\left(  A\otimes B\right)  ^{\#}=B^{\#}\otimes^{\prime}A^{\#}\label{AB hash}%
\end{equation}
and%
\begin{equation}
\left(  A\otimes^{\prime}B\right)  ^{\#}=B^{\#}\otimes A^{\#}\label{AdashB}%
\end{equation}
whenever $A$ and $B$ are compatible. Further, for any $u\in\mathbb{R}^{n}$ we
have
\begin{equation}
u^{\#}\otimes u=0\label{u hash u}%
\end{equation}
and
\begin{equation}
u\otimes u^{\#}\geq I.\label{u u hash}%
\end{equation}

We will usually not write the operator $\otimes^{\prime}$and for matrices the
convention applies that if no operator appears then the product is in
min-algebra whenever it follows the symbol $\#$, otherwise it is in
max-algebra. In this way a residuated pair of operations (a special case of
Galois connection) has been defined, namely%
\begin{equation}
Ax\leq y\Longleftrightarrow x\leq A^{\#}y\label{residuation}%
\end{equation}
for all $x,y\in\mathbb{R}^{n}.$ Hence $Ax\leq y$ implies $A(A^{\#}y)\leq y$.
It follows immediately that a one-sided system $Ax=b$ has a solution if and
only if $A\left(  A^{\#}b\right)  =b$ and that the system $Ax\leq b$ always
has an infinite number of solutions with $A^{\#}b$ being the greatest solution.

\section{Duality for tropical linear programs\label{Sec LP duality}}

\bigskip Let $A=\left(  a_{ij}\right)  \in\mathbb{R}^{m\times n},b=\left(
b_{1},...,b_{m}\right)  ^{T}\in\mathbb{R}^{m},c=\left(  c_{1},...,c_{n}%
\right)  ^{T}\in\mathbb{R}^{n}$ and consider the following primal-dual pair of
tropical linear programs:%
\begin{equation}
\left.
\begin{array}
[c]{c}%
f\left(  x\right)  =c^{T}x\rightarrow\max\\
s.t.\\
Ax\leq b\\
x\in\mathbb{R}^{n}%
\end{array}
\right\} \tag*{(P)}%
\end{equation}
\label{(P)}

and%
\begin{equation}
\left.
\begin{array}
[c]{c}%
\varphi\left(  \pi\right)  =\pi^{T}b\rightarrow\min\\
s.t.\\
\pi^{T}A\geq c^{T}\\
\pi\in\mathbb{R}^{m}%
\end{array}
\right\} \tag*{(D)}\label{(D)}%
\end{equation}

\bigskip Let us denote $S_{P}=\left\{  x\in\mathbb{R}^{n}:Ax\leq b\right\}  $
and $S_{D}=\left\{  \pi\in\mathbb{R}^{m}:\pi^{T}A\geq c^{T}\right\}  .$ The
sets of optimal solutions will be denoted $S_{P}^{opt}$ and $S_{D}^{opt},$
respectively. The optimal objective function values will be denoted $f^{\max}$
and $\varphi^{\min}.$

\begin{theorem}
[Weak Duality Theorem]\cite{hoffman} The inequality $c^{T}x\leq\pi^{T}b$ holds
for any $x\in S_{P}$ and $\pi\in S_{D}.$
\end{theorem}

\begin{proof}
By isotonicity and associativity we have
\[
c^{T}x\leq\left(  \pi^{T}A\right)  x=\pi^{T}\left(  Ax\right)  \leq\pi^{T}b.
\]

\end{proof}

\begin{theorem}
[Strong Duality Theorem]\cite{superville} Optimal solutions to both (P) and
(D) exist and%
\[
\max_{x\in S_{P}}c^{T}x=\min_{\pi\in S_{D}}\pi^{T}b.
\]

\end{theorem}

We first prove a lemma:

\begin{lemma}
$A^{\#}b\in S_{P}^{opt}$ and hence $f^{\max}=c^{T}\left(  A^{\#}b\right)  .$
\end{lemma}

\begin{proof}
By residuation (\ref{residuation}) we have%
\[
x\in S_{P}\Longleftrightarrow x\leq A^{\#}b
\]
and the rest follows by isotonicity of $\otimes.$
\end{proof}

\begin{proof}
(Of strong duality.) Let us denote $t=c^{T}\left(  A^{\#}b\right)  $ and set
$\pi^{T}=tb^{\#}.$ It remains to prove that $\pi$ is dual feasible and
$\pi^{T}b=t.$ Using associativity and isotonicity (\ref{Isotonicity}) and also
(\ref{AdashB}) and (\ref{u u hash}) we have%
\begin{align*}
\pi^{T}A  & =\left(  t\otimes b^{\#}\right)  \otimes A\\
& =t\otimes\left(  b^{\#}\otimes A\right) \\
& =c^{T}\otimes\left(  A^{\#}\otimes^{\prime}b\right)  \otimes\left(
A^{\#}\otimes^{\prime}b\right)  ^{\#}\\
& \geq c^{T}\otimes I\\
& =c^{T}.
\end{align*}
On the other hand (using (\ref{u hash u}))%
\[
\pi^{T}b=\left(  t\otimes b^{\#}\right)  \otimes b=t\otimes\left(
b^{\#}\otimes b\right)  =t\otimes0=t,
\]
which completes the proof.
\end{proof}

It follows that there is no duality gap for the pair of programs (P) and (D).
Their optimal solutions are $\overline{x}=A^{\#}b$ (no matter what $c$ is) and
$\overline{\pi}^{T}=c^{T}\left(  A^{\#}b\right)  b^{\#},$ respectively. Their
common objective function value is $c^{T}\left(  A^{\#}b\right)  .$ If $A,b,c$
are integer then there is also no duality gap for the corresponding pair of
integer programs since then both $\overline{x}$ and $\overline{\pi}$ are
integer. However, this is not the case when some of $A,b,c$ are non-integer
and the question of a duality gap arises. This will be discussed in the next section.

\section{Dual integer programs\label{Sec Integer duality}}

\bigskip Let $A=\left(  a_{ij}\right)  \in\mathbb{R}^{m\times n},b=\left(
b_{1},...,b_{m}\right)  ^{T}\in\mathbb{R}^{m},c=\left(  c_{1},...,c_{n}%
\right)  ^{T}\in\mathbb{R}^{n}$ and consider the following pair of tropical
integer programs:%
\begin{equation}
\left.
\begin{array}
[c]{c}%
f\left(  x\right)  =c^{T}x\rightarrow\max\\
s.t.\\
Ax\leq b\\
x\in\mathbb{Z}^{n}%
\end{array}
\right\} \tag*{(PI)}%
\end{equation}
\label{(P) copy(1)}

and%
\begin{equation}
\left.
\begin{array}
[c]{c}%
\varphi\left(  \pi\right)  =\pi^{T}b\rightarrow\min\\
s.t.\\
\pi^{T}A\geq c^{T}\\
\pi\in\mathbb{Z}^{m}%
\end{array}
\right\} \tag*{(DI)}%
\end{equation}

\bigskip Let us denote $S_{PI}=\left\{  x\in\mathbb{Z}^{n}:Ax\leq b\right\}  $
and $S_{DI}=\left\{  \pi\in\mathbb{Z}^{m}:\pi^{T}A\geq c^{T}\right\}  .$ The
sets of optimal solutions will be denoted $S_{PI}^{opt}$ and $S_{DI}^{opt},$
respectively. The optimal objective function values will be denoted
$f_{I}^{\max}$ and $\varphi_{I}^{\min}.$

It follows immediately from residuation that $S_{PI}=\left\{  x\in
\mathbb{Z}^{n}:x\leq\left\lfloor A^{\#}b\right\rfloor \right\}  .$ \ Hence by
isotonicity $\left\lfloor A^{\#}b\right\rfloor \in S_{PI}^{opt}$ and
$f_{I}^{\max}=c^{T}\left\lfloor A^{\#}b\right\rfloor .$

Solving (DI) is less straightforward but can be done directly when
$b\in\mathbb{Z}^{m}.$ This will be shown below, but first we answer a
principle question for general (real) $A,b$ and $c.$

\begin{proposition}
\label{Prop Opt sol exists}$S_{DI}^{opt}\neq\emptyset.$
\end{proposition}

\begin{proof}
Since $S_{DI}$ is a closed, non-empty set and $\varphi$ is continuous and
bounded below due to the Weak Duality Theorem we only need to prove that
$S_{DI}$ can be restricted to a bounded set without affecting $\varphi
_{I}^{\min}.$ Let $U=\varphi\left(  \pi_{0}\right)  $ for an arbitrarily
chosen $\pi_{0}\in S_{DI}$ and $L$ be any lower bound following from the Weak
and Strong Duality Theorems. Then disregarding $\pi\in S_{DI}$ with $\pi
_{i}>U-b_{i}$ for at least one $i$ will have no affect on $\varphi_{I}^{\min
}.$ Similarly disregarding those $\pi$ (if any) where $\pi_{i}<L-b_{i}$ for at
least one $i$. Hence for solving (DI) we\ may assume without loss of
generality that%
\[
B^{-1}L\leq\pi\leq B^{-1}U,
\]
where $B=diag\left(  b\right)  .$ The feasible set in (DI) is now restricted
to a compact set and the statement follows.
\end{proof}

We will transform (DI) to an equivalent "normalised" tropical integer program.
Let us denote $B=diag\left(  b\right)  ,$ $C=diag\left(  c\right)  $ and
$\sigma^{T}=\pi^{T}B.$ Hence $\pi^{T}=\sigma^{T}B^{-1},\pi^{T}b=\sigma^{T}0$
and the inequality in (DI) is equivalent to%
\begin{equation}
\sigma^{T}B^{-1}AC^{-1}\geq0\label{eq   new ineq}%
\end{equation}
or, component-wise:%
\[
\max_{i}\left(  \sigma_{i}-b_{i}+a_{ij}-c_{j}\right)  \geq0\text{ \ for all
}j.
\]
Let us denote the matrix $B^{-1}AC^{-1}$ by $D.$ The new tropical integer
program is%
\begin{equation}
\left.
\begin{array}
[c]{c}%
\varphi^{\prime}\left(  \sigma\right)  =\sigma^{T}0\rightarrow\min\\
s.t.\\
\sigma^{T}D\geq0^{T}\\
\sigma\in\mathbb{Z}^{m}.
\end{array}
\right\} \tag{DI'}%
\end{equation}

\begin{proposition}
If $b\in\mathbb{Z}^{m}$ then (DI) and (DI') are equivalent and a one-to-one
correspondence between feasible solutions of these problems is given by
$\pi^{T}=\sigma^{T}B^{-1}.$
\end{proposition}

\begin{proof}
If $b\in\mathbb{Z}^{m}$ then $\pi$ is integer if and only if $\sigma$ is
integer; the rest follows from the previous discussion.
\end{proof}

\begin{proposition}
\bigskip\label{Prop Constant vector}If%
\[
t=\min_{\sigma\in S_{DI^{\prime}}}\varphi^{\prime}\left(  \sigma\right)
\]
then $\sigma_{0}=\left(  t,...,t\right)  ^{T}$ is an optimal solution to (DI').
\end{proposition}

\begin{proof}
\bigskip Let $t=\min_{\sigma\in S_{DI^{\prime}}}\varphi^{\prime}\left(
\sigma\right)  =\varphi^{\prime}\left(  \widetilde{\sigma}\right)  $ for some
$\widetilde{\sigma}\in S_{DI^{\prime}}.$ Then $t=\widetilde{\sigma}^{T}%
0=\max\left(  \widetilde{\sigma}_{1},...,\widetilde{\sigma}_{m}\right)
\in\mathbb{Z}$ and so we have $\sigma_{0}\geq\widetilde{\sigma},$ hence
$\sigma_{0}^{T}D\geq\widetilde{\sigma}^{T}D\geq0^{T},\sigma_{0}\in
\mathbb{Z}^{m}$ and $\varphi^{\prime}\left(  \sigma_{0}\right)  =t.$ The
statement follows.
\end{proof}

By Proposition \ref{Prop Constant vector} we may restrict our attention to
constant vectors when searching for optimal solutions of (DI'). If $\sigma
^{T}=\left(  s,...,s\right)  $ then (\ref{eq new ineq}) reads%

\[
\max_{i}\left(  s-b_{i}+a_{ij}-c_{j}\right)  \geq0\text{ \ for all }j,
\]
equivalently,%
\[
s\geq\min_{i}\left(  b_{i}-a_{ij}+c_{j}\right)  \text{ \ for all }j,
\]
or,%
\[
s\geq\max_{j}\min_{i}\left(  b_{i}-a_{ij}+c_{j}\right)  .
\]
This can also be written as%
\begin{align*}
s  & \geq\max_{j}\left(  c_{j}+\min_{i}\left(  b_{i}-a_{ij}\right)  \right) \\
& =\max_{j}\left(  c_{j}+\min_{i}\left(  a_{ji}^{\#}+b_{i}\right)  \right) \\
& =c^{T}\left(  A^{\#}b\right)  .
\end{align*}

Since $t$ is the minimal possible integer value of $s,$ we have%
\[
t=\left\lceil c^{T}\left(  A^{\#}b\right)  \right\rceil .
\]
We have proved:

\begin{proposition}
If $b\in\mathbb{Z}^{m}$ then $\min_{\sigma\in S_{DI^{\prime}}}\varphi^{\prime
}\left(  \sigma\right)  =\left\lceil c^{T}\left(  A^{\#}b\right)  \right\rceil
$ and $\sigma=\left(  t,...,t\right)  ^{T}$ is an optimal solution of (DI'),
where $t=\left\lceil c^{T}\left(  A^{\#}b\right)  \right\rceil $.
\end{proposition}

We also conclude that $\pi^{T}=t0^{T}B^{-1}=tb^{\#}$ is an optimal solution of
(DI) with $\varphi_{I}^{\min}=\pi^{T}b=t\otimes b^{\#}\otimes b=t0=t.$
Therefore the duality gap for the pair (PI) - (DI) when $b\in\mathbb{Z}^{m}$
is the interval%
\[
\left(  c^{T}\left\lfloor A^{\#}b\right\rfloor ,\left\lceil c^{T}\left(
A^{\#}b\right)  \right\rceil \right)  .
\]

Solution of dual integer programs without the assumption $b\in\mathbb{Z}^{m} $
is presented in Chapter \ref{Sec general dual}.

\section{Using duality for solving a two-sided linear program}

We will now use the results of Section \ref{Sec LP duality} to directly solve
a special tropical linear program with two-sided constraints.

Consider the two-sided tropical linear program%
\begin{equation}
\left.
\begin{array}
[c]{c}%
g\left(  y\right)  =c^{T}y\rightarrow\min\\
s.t.\\
Ay\oplus d\leq y\\
y\in\mathbb{R}^{n}%
\end{array}
\right\} \tag*{(TSLP)}%
\end{equation}
where $A\in\mathbb{R}^{n\times n}$ and $d\in\mathbb{R}^{n}.$ The inequality
$Ay\oplus d\leq y$ is equivalent to the following system of inequalities:%
\begin{align*}
Ay  & \leq y,\\
d  & \leq y.
\end{align*}
The first of these inequalities $Ay\leq y$ describes the set of (finite)
subeigenvectors of $A$ (see Section \ref{Sec Prelim}) corresponding to
$\lambda=0,$ that is $V_{\ast}(A,0).$ By Theorem
\ref{Th Set of subeigenvectors} this set is non-empty if and only if
$\lambda\left(  A\right)  \leq0.$ Therefore we suppose now that this condition
is satisfied. By homogeneity of $Ay\leq y$ we may assume that a subeigenvector
sufficiently large exists and in particular one that also satisfies $y\geq d.
$ This implies that the feasible set of TSLP is nonempty.

\bigskip Let us denote%
\[
S=\left\{  y\in\mathbb{R}^{n}:Ay\oplus d\leq y\right\}  .
\]
By Theorem \ref{Th Set of subeigenvectors} we have%
\[
S=\left\{  y=A^{\ast}u:u\in\mathbb{R}^{n},A^{\ast}u\geq d\right\}  ,
\]
where $A^{\ast}$ is the Kleene Star defined by (\ref{Def Delta}).

\bigskip Hence $g\left(  y\right)  =c^{T}A^{\ast}u$ and (TSLP) now reads (in
the form of a dual problem):%
\begin{equation}
\left.
\begin{array}
[c]{c}%
h\left(  u\right)  =u^{T}\left(  A^{\ast^{T}}c\right)  \rightarrow\min\\
s.t.\\
u^{T}A^{\ast^{T}}\geq d^{T}\\
u\in\mathbb{R}^{n}%
\end{array}
\right\} \tag*{(TSLP')}%
\end{equation}

\bigskip In order to use the results of Section \ref{Sec LP duality} we
substitute as follows: $\pi\rightarrow u,$ $\varphi\rightarrow h,$
$b\rightarrow A^{\ast^{T}}c,$ $A\rightarrow A^{\ast^{T}},$ $c\rightarrow d.$
This yields the vector%
\[
\overline{y}=A^{\ast}\overline{u},
\]
as an optimal solution of (TSLP) where%
\[
\overline{u}=d^{T}\otimes\left(  -A^{\ast}\otimes^{\prime}\left(  A^{\ast^{T}%
}\otimes c\right)  \right)  \otimes\left(  c^{\#}\otimes^{\prime}\left(
-A^{\ast}\right)  \right)  .
\]
The optimal objective function value is
\[
g^{\min}=d^{T}\otimes\left(  -A^{\ast}\otimes^{\prime}\left(  A^{\ast^{T}%
}\otimes c\right)  \right)  .
\]

Computationally the most demanding part here is the calculation of $A^{\ast} $
which is $O\left(  n^{3}\right)  $. \bigskip We conclude that (TSLP) can be
solved directly in $O\left(  n^{3}\right)  $ time.

We finish this section by a remark on a tropical linear program obtained from
TSLP by replacing the inequalities with equations:%

\begin{equation}
\left.
\begin{array}
[c]{c}%
g\left(  y\right)  =c^{T}y\rightarrow\min\\
s.t.\\
Ay\oplus d=y\\
y\in\mathbb{R}^{n}%
\end{array}
\right\} \tag*{(TSLP2)}%
\end{equation}
It is known \cite{BSS} that if we denote%
\[
S=\left\{  y:Ay\oplus d=y\right\}
\]
then assuming $\lambda\left(  A\right)  \leq0$ again we have%
\[
S=\left\{  y=v\oplus A^{\ast}d:Av=v\right\}
\]
and thus%
\[
\min_{y\in S}g\left(  y\right)  =\min_{v\in\mathbb{R}^{n}}\left\{
c^{T}v\oplus c^{T}A^{\ast}d:Av=v\right\}  =c^{T}A^{\ast}d
\]
since for sufficiently small $v$ (which can be assumed by homogeneity of
$Av=v$) we have $c^{T}v\leq c^{T}A^{\ast}d.$

Note that if $\lambda\left(  A\right)  <0$ then $S=\left\{  A^{\ast}d\right\}
$ and so (TSLP2) has a non-trivial set of feasible solutions if and only if
$\lambda\left(  A\right)  =0.$

\section{An algorithm for general integer dual
programs\label{Sec general dual}}

The explicit solution of (DI) in Section \ref{Sec Integer duality} depends on
the assumption $b\in\mathbb{Z}^{m}.$ If $b$ is non-integer then Proposition
\ref{Prop Constant vector} is not available and it is not clear whether a
direct solution method can be produced. Therefore we now present an algorithm
for solving (DI) without any assumption on $b$ other than $b\in\mathbb{R}%
^{m}.$ Note that existence of a lower bound of the objective function follows
from the Weak Duality Theorem immediately.

First we repeat the transformation to a "normalised" tropical linear program
(DI'). As before we denote $B=diag\left(  b\right)  ,$ $C=diag\left(
c\right)  $ and $\sigma^{T}=\pi^{T}B.$ Hence $\pi^{T}=\sigma^{T}B^{-1},\pi
^{T}b=\sigma^{T}0$ and the inequality in (DI) is equivalent to%
\begin{equation}
\sigma^{T}B^{-1}AC^{-1}\geq0\label{eq   new ineq 2}%
\end{equation}
or, component-wise:%
\begin{equation}
\max_{i}\left(  \sigma_{i}-b_{i}+a_{ij}-c_{j}\right)  \geq0\text{ \ for all
}j.\label{eq  new ineq 3}%
\end{equation}
\ As before let us denote the matrix $B^{-1}AC^{-1}$ by $D.$ The new tropical
program is%
\begin{equation}
\left.
\begin{array}
[c]{c}%
\varphi^{\prime}\left(  \sigma\right)  =\sigma^{T}0\rightarrow\min\\
s.t.\\
\sigma^{T}D\geq0^{T}.
\end{array}
\right\} \tag{DI2'}%
\end{equation}
Now we cannot assume $\sigma\in\mathbb{Z}^{m}$ since the inverse
transformation $\pi^{T}=\sigma^{T}B^{-1}$ would not produce $\pi\in
\mathbb{Z}^{m}$ in general. It is also not sufficient to require\ $\sigma
\in\mathbb{R}^{m}$ in order to obtain $\pi\in\mathbb{Z}^{m}.$ However, since
$\sigma_{i}=\pi_{i}+b_{i},\pi_{i}\in\mathbb{Z}$ for every $i$ we have that
$\sigma$ should satisfy%
\[
fr\left(  \sigma_{i}\right)  =fr\left(  b_{i}\right)  \text{ \ for all }i.
\]
In order to meet this requirement we introduce (for a given $b\in
\mathbb{R}^{m}$) real functions $\left\lceil .\right\rceil ^{\left(  i\right)
}$ ($i=1,...,m$) as follows: if $x\in\mathbb{R}$ then $\left\lceil
x\right\rceil ^{\left(  i\right)  }$ is the least real number $u$ such that
$x\leq u$ and $fr\left(  u\right)  =fr\left(  b_{i}\right)  .$

Condition (\ref{eq new ineq 3}) can be stated as follows:%
\[
\left(  \forall j\right)  \left(  \exists i\right)  \left(  \sigma_{i}%
-b_{i}+a_{ij}-c_{j}\geq0\right)
\]
or, equivalently%
\[
\left(  \forall j\right)  \left(  \exists i\right)  \left(  \sigma_{i}\geq
b_{i}-a_{ij}+c_{j}\right)
\]
and taking into account desired integrality of $\pi:$%
\begin{equation}
\left(  \forall j\right)  \left(  \exists i\right)  \left(  \sigma_{i}%
\geq\left\lceil -d_{ij}\right\rceil ^{\left(  i\right)  }\right)
,\label{eq  condition 1}%
\end{equation}
where $D=\left(  d_{ij}\right)  .$ Because of the minimization of the
objective function every component $\sigma_{i}$ of an optimal solution
$\sigma$ may be assumed to be actually equal to $\left\lceil -d_{ij}%
\right\rceil ^{\left(  i\right)  }$ for at least one $j$ or to $\left\lceil
L\right\rceil ^{\left(  i\right)  }$ where $L$ is any lower bound of
$\varphi^{\prime}\left(  \sigma\right)  $ in (DI2'). Conversely any $\sigma$
satisfying (\ref{eq condition 1}) where for every $i$ equality is attained for
at least one $j$ or $\sigma_{i}=\left\lceil L\right\rceil ^{\left(  i\right)
}$ produces an integer solution $\pi$ of (DI) using the transformation
$\pi^{T}=\sigma^{T}B^{-1}.$

Let us denote for $\sigma\in\mathbb{R}^{m}$ and $i=1,...,m:$%
\[
N_{i}\left(  \sigma\right)  =\left\{  j\in N:\sigma_{i}\geq\left\lceil
-d_{ij}\right\rceil ^{\left(  i\right)  }\right\}  .
\]
We can summarize our discussion as follows:

\begin{proposition}
\label{Prop Feasibility condition}The vector $\pi^{T}=\sigma^{T}B^{-1}$ is a
feasible solution of (DI) only if%
\[
\bigcup\nolimits_{i=1,...,m}N_{i}\left(  \sigma\right)  =N.
\]
There is an optimal solution $\pi$ such that the vector $\sigma^{T}=\pi^{T}B$
also satisfies%
\[
\left(  \forall i\right)  \left(  \exists j\right)  \left(  \sigma
_{i}=\left\lceil -d_{ij}\right\rceil ^{\left(  i\right)  }\text{ or
}\left\lceil L\right\rceil ^{\left(  i\right)  }\right)  .
\]

\end{proposition}

Proposition \ref{Prop Feasibility condition} enables us to compile the
following algorithm for finding an optimal solution of (DI) for general (real)
entries $A,b$ and $c.$ Here we denote%
\[
M_{ij}=\left\lceil -d_{ij}\right\rceil ^{\left(  i\right)  }\text{ \ for every
}i\text{ and }j
\]
and%
\[
M_{i,n+1}=\left\lceil L\right\rceil ^{\left(  i\right)  }\text{ for every }i.
\]

ALGORITHM

\begin{enumerate}
\item $\sigma_{i}:=\max_{j=1,...,n+1}M_{ij}$ for $i=1,...,m.$

\item $K:=\left\{  i\in M:\varphi^{\prime}\left(  \sigma\right)  =\sigma
_{i}\neq\left\lceil L\right\rceil ^{\left(  i\right)  }\right\}  .$

\item (a) For all $i\in K$ set $\sigma_{i}^{\prime}:=\max_{j}\left\{
M_{ij}:M_{ij}<\sigma_{i}\right\}  .$

(b) For all $i\notin K$ set $\sigma_{i}^{\prime}:=\sigma_{i}$.

\item If $\bigcup\nolimits_{i=1,...,m}N_{i}\left(  \sigma^{\prime}\right)
\neq N$ or $K=\emptyset$ then stop ($\sigma$ is optimal).

\item $\sigma:=\sigma^{\prime}$

\item Go to 2.
\end{enumerate}

The number of iterations of this algorithm does not exceed $mn$ since each of
$m$ variables $\sigma_{i}$ can decrease at most $n$ times. The number of
operations in each iteration is $O\left(  m\right)  $ in steps 2,3 and 5 and
$O\left(  mn\right)  $ in step 4. Hence the algorithm is $O\left(  m^{2}%
n^{2}\right)  .$ This includes a possible pre-ordering of each of the sets
$\left\{  M_{ij}:j=1,...,n+1\right\}  ,i=1,...,m$ which is $O\left(  mn\log
n\right)  $ but can be done once before the start of the main loop.

Note that by taking $\left\lfloor b\right\rfloor $ for $b$ we get (DI) that
can be solved directly (Section \ref{Sec Integer duality}) and the difference
between $\varphi_{I}^{\min}$ and $\varphi^{\min}$ is up to 1 since
$\left\lfloor \pi^{T}b\right\rfloor =\pi^{T}\left\lfloor b\right\rfloor $ if
$\pi$ is integer. Hence a good estimate of the optimal objective function
value can be obtained by directly solving (DI) where $b$ is replaced by
$\left\lfloor b\right\rfloor $.

\section{Conclusions}

We have presented simple proofs of the strong and weak duality theorems of
tropical linear programming. It follows that there is no duality gap in
tropical linear programming. This result together with known results on
subeigenvectors enables us to solve a special tropical linear program with
two-sided constraints in $O\left(  n^{3}\right)  $ time.

We have then studied the duality gap in tropical integer linear programming. A
direct solution is available for the primal problem. An algorithm of quadratic
complexity has been presented for the dual problem. A direct solution of the
dual problem is available provided that all coefficients of the objective
function are integer. This solution readily provides a good estimate of the
optimal objective function value for general dual integer programs.

\textbf{Acknowledgement}: This work was supported by the EPSRC grant EP/J00829X/1.

\end{document}